\documentclass[12pt]{amsart}
\usepackage{amssymb}
\usepackage{amsxtra}
\usepackage{stmaryrd}
\usepackage[all]{xy}
\usepackage{graphics}

\newtheorem{thm}{Theorem}[section]
\newtheorem{lem}[thm]{Lemma}
\newtheorem{cor}[thm]{Corollary}
\newtheorem{prop}[thm]{Proposition}
\newtheorem{ex}[thm]{Example}

\newtheorem*{prob*}{Open problem}

\theoremstyle{definition}

\newtheorem{defi}[thm]{Definition}

\theoremstyle{remark}

\newtheorem{rem}[thm]{Remark}
\newtheorem*{rem*}{Remark}

\DeclareMathOperator{\Hom}{Hom}

\newcommand{\kringel}{\mathbin{\raise1pt\hbox{$\scriptstyle\circ$}}} 
\newcommand{\pkt}{\mathbin{\raise0pt\hbox{$\scriptstyle\bullet$}}}

\newcommand{\C}{\mathbb{C}}

\newcommand{\N}{\mathbb{N}}

\newcommand{\Z}{\mathbb{Z}}

\newcommand{\tr}{\mathop{\rm tr}}

\newcommand{\End}{\mathop{\rm End}}
\newcommand{\Der}{\mathop{\rm Der}}

\newcommand{\Lg}{\mathfrak{g}}

\newcommand{\Lr}{\mathfrak{r}}

\newcommand{\CL}{\mathcal{L}}

\newcommand{\CP}{\mathcal{P}}
\newcommand{\CR}{\mathcal{R}}

\newcommand{\CZ}{\mathcal{Z}}

\newcommand{\rank}{{\rm rank}}

\newcommand{\al}{\alpha}
\newcommand{\be}{\beta}
\newcommand{\ga}{\gamma}

\newcommand{\ep}{\varepsilon}

\newcommand{\la}{\lambda}

\newcommand{\ov}{\overline}

\newcommand{\ra}{\rightarrow}

\renewcommand{\phi}{\varphi}

\begin{document}

\title[Degenerations]{Degenerations of pre-Lie algebras}

\author[T. Bene\v{s}]{Thomas Bene\v{s}}
\author[D. Burde]{Dietrich Burde}

\address{Fakult\"at f\"ur Mathematik\\
Universit\"at Wien\\
  Nordbergstr. 15\\
  1090 Wien \\
  Austria} 
\email{thomas.benes@univie.ac.at}
\email{dietrich.burde@univie.ac.at}
\date{\today}


\begin{abstract}
We consider the variety of pre-Lie algebra structures on a given $n$-dimensional vector space.
The group $GL_n(K)$ acts on it, and we study the closure of the orbits with respect to the 
Zariski topology. This leads to the definition of pre-Lie algebra degenerations. We give
fundamental results on such degenerations, including invariants and necessary degeneration 
criteria. We demonstrate the close relationship to Lie algebra degenerations.
Finally we classify  all orbit closures in the variety of complex $2$-dimensional pre-Lie 
algebras. 
\end{abstract}

\maketitle

\section{Introduction}

Contractions of Lie algebras are limiting processes between Lie algebras, which have
been studied first in physics \cite{SEG},\cite{IWI}. For example, classical mechanics
is a limiting case of quantum mechanics as $\hbar \ra 0$, described by a contraction of the
Heisenberg-Weyl Lie algebra to the abelian Lie algebra of the same dimension. \\
In mathematics, often a more general definition of contractions is used, so called 
degenerations.
Here one considers the variety of $n$-dimensional Lie algebra structures and the
orbit closures with respect to the Zariski topology of $GL_n(K)$-orbits. There is a large
literature on degenerations, see for example \cite{NEP} and the references cited therein.
Degenerations have been studied also for commutative algebras, associative algebras and
Leibniz algebras. Of course, orbit closures and hence degenerations can be considered for
all algebras. However, we are particularly interested in so called {\it pre-Lie algebras},
which have very interesting applications in geometry and physics, see \cite{BU2} for a survey.
This class of algebras also includes Novikov algebras. For applications of Novikov
algebras in physics, see \cite{BAL}.\\
The aim of this article is to provide a degeneration theory for pre-Lie algebras, and to
find interesting invariants, which are preserved under the process of degeneration.
It turns out, that among other things such invariants are given by
polynomial operator identities $T(x,y)=0$ in the operators
$L(x),L(y),R(x),R(y)$, the left and right multiplications of the pre-Lie algebra.
For example, the identity $T(x,y)=L(x)R(y)-R(y)L(x)=0$ for all $x,y\in A$ says 
that the algebra $A$ is associative. This is preserved under degeneration. \\
On the other hand we find semi-continuous functions on the variety of $n$-dimensional 
pre-Lie algebra structures. An example is given by the dimension of the center
$\CZ(A)$ of a pre-Lie algebra. If a pre-Lie algebra $A$ degenerates to a pre-Lie
algebra $B$, then $\dim \CZ(A)\le \dim\CZ(B)$. The function $f(\la)=\dim \CZ(\la)$ is
an upper semi-continuous function on  the variety of pre-Lie algebra structures. 
We may also consider the dimensions of left and right annihilators, or in fact
of various other spaces, like certain subalgebras and cohomology spaces.
Since a pre-Lie algebra in general is not anti-commutative, we often have two possibilities
(like right and left annihilators), where we had only one in the Lie algebra case. \\
Furthermore we introduce generalized derivation algebras. The dimension of these
spaces are again upper semi-continuous functions. More generally, certain generalized 
cohomology spaces can be studied in this context. \\
Finally we apply our results to classify all orbit closures in the variety of 
$2$-dimensional pre-Lie algebras. The resulting Hasse diagram shows how 
complicated the situation is already in dimension $2$.

\section{The variety of pre-Lie algebras}

Pre-Lie algebras, or left-symmetric algebras arise in many areas of mathematics
and physics. It seems that A. Cayley in $1896$ was the first one to introduce pre-Lie 
algebras, in the context of rooted tree algebras. From $1960$ onwards they became
widely known by the work of Vinberg, Koszul and Milnor in connection with
convex homogeneous cones and affinely flat manifolds. Around $1990$ they appeared in
renormalization theory and quantum mechanics, starting with the work of Connes and Kreimer.
For the details and the references see \cite{BU2}. The definition is as follows:

\begin{defi}
A $K$-algebra $A$ together with a bilinear product $(x,y)\mapsto x\cdot y$ is called
a {\it pre-Lie algebra}, if the identity
\begin{align}\label{lsa}
(x\cdot y) \cdot z- x \cdot (y \cdot z) & = (y\cdot x) \cdot z- y \cdot (x \cdot z)
\end{align}
holds for all $x,y,z\in A$. A pre-Lie algebra is called a {\it Novikov algebra}, if the
identity
\begin{equation}\label{nov}
(x\cdot y)\cdot z=(x\cdot z)\cdot y
\end{equation}
holds for all $x,y,z\in A$.
\end{defi}

The commutator $[x,y]=x\cdot y-y\cdot x$ in a pre-Lie algebra defines a Lie bracket.
We denote the associated Lie algebra by $\Lg_A$.

\begin{rem}
If $(x,y,z)=(x\cdot y) \cdot z- x \cdot (y \cdot z)$ denotes the associator in $A$, then
$A$ is a pre-Lie algebra, if and only if $(x,y,z)=(y,x,z)$ for all $x,y,z$ in $A$. 
For this reason, $A$ is also called a {\it left-symmetric algebra}.
\end{rem}

In analogy to the variety $\CL_n(K)$ of $n$-dimensional Lie algebra structures 
we can define the variety of arbitrary non-associative algebras. 
We want to focus here on pre-Lie algebras. Let $V$ be a vector space of 
dimension $n$ over a field $K$. Fix a basis $(e_1,\ldots,e_n)$ of $V$. 
If $(x,y)\mapsto x\cdot y$ is a pre-Lie algebra product on $V$ with 
$e_i\cdot e_j= \sum_{k=1}^n c_{ij}^k e_k$, then $(c_{ij}^k)\in K^{n^3}$
is called a {\it pre-Lie algebra structure} on $V$.

\begin{defi}
Let $V$ be an $n$-dimensional vector space over a field $K$. 
Denote by $\CP_n(K)$ the set of all pre-Lie algebra structures on $V$.
This is called the {\it variety of pre-Lie algebra structures}.
\end{defi}

The set $\CP_n(K)$ is an affine algebraic set, since the identity \eqref{lsa} is given by
polynomials in the structure constants $c_{ij}^k$. It need not be irreducible, however.
Denote by $\mu$ a pre-Lie algebra product on $V$. 
The general linear group $GL_n(K)$ acts on $\CP_n(K)$ by
\begin{equation*}
(g\kringel \mu)(x,y)=g(\mu(g^{-1}x, g^{-1}y))
\end{equation*}
for $g\in GL_n(K)$ and $x,y\in V$. \\
Denote by $O(\mu)$ the orbit of $\mu$ under this action, and by $\ov{O(\mu)}$ the closure 
of the orbit with respect to the Zariski topology. 
The orbits in $\CP_n(K)$ under this action correspond to isomorphism classes of
$n$-dimensional pre-Lie algebras. 

\begin{defi}
Let $\la,\mu \in \CP_n(K)$ be two pre-Lie algebra laws.
We say that $\la$ {\it degenerates} to $\mu$, if $\mu \in \ov{O(\la)}$. This is denoted by 
$\la \ra_{\rm deg} \mu$. We say that the degeneration $\la \ra_{\rm deg} \mu$ is {\it proper}, 
if $\mu \in \ov{O(\la)}\setminus O(\la)$, i.e., if $\la$ and $\mu$ are not isomorphic.
\end{defi}

The existence of a pre-Lie algebra degeneration $A\ra_{\rm deg}B$ means the following:
the algebra $B$ is represented by a structure $\mu$ which
lies in the Zariski closure of the $GL_n(K)$-orbit of some structure $\la$ which
represents $A$. 

The following important result is due to Borel \cite{BOR}:

\begin{prop}
If $G$ is a complex algebraic group and $X$ is a complex algebraic variety
with regular action, then each orbit $G(x)$, $x\in X$ is a smooth algebraic variety, 
open in its closure $\ov{G(x)}$. Its boundary $\ov{G(x)}\setminus G(x)$ is a union of 
orbits of strictly lower dimension. Each orbit $G(x)$ is a constructible set, 
hence $\ov{G(x)}$ coincides with the closure $\ov{G(x)}^d$ in the standard Euclidean topology. 
\end{prop}

Recall that a subset $Y\subseteq X$ is called constructible if it is a finite union of
locally closed sets. The result has some interesting consequences:

\begin{cor}
Denote by $\C(t)$ the field of fractions of the polynomial ring $\C[t]$. If there
exists an operator $g_t \in GL_n(\C(t))$ such that
$\lim_{t \to 0}g_t \kringel \la=\mu$ for $\la,\mu\in \CP_n(\C)$, then $\la \ra_{\rm deg} \mu$.
\end{cor}

\begin{proof}
We have $\mu\in \ov{O(\la)}^d$ by assumption, which implies $\mu\in \ov{O(\la)}$.
\end{proof}

\begin{ex}
Any $n$-dimensional complex pre-Lie algebra $\la$ degenerates to the
zero pre-Lie algebra $\C^n$.
\end{ex}

Let $g_t=t^{-1}E_n$, where $E_n$ is the identity matrix. Then we have 
\[
(g_t\kringel \la)(x,y)=t^{-1}\la(tx,ty)=t \la(x,y),
\]
hence $\la$ degenerates to the zero product, i.e., $\lim_{t \to 0}g_t \kringel \la=\C^n$.

\begin{rem}\label{2.8}
Borel's result implies also the following (the argument is the same as the one
given in \cite{NEP} for Lie algebras).
Every degeneration of complex Novikov algebras can be realized by a so called sequential 
contraction, i.e., $A\ra_{\rm deg} B$ is equi\-valent to the fact that we have
\[
\lim_{\ep \to 0}g_{\ep} \kringel \la=\mu
\]
where $g_{\ep}\in GL_n(\C)$, $\ep>0$ and $\la,\mu\in \CP_n(\C)$, such that $\la$ and 
$\mu$ represent the pre-Lie algebras $A$ and $B$ respectively. 
\end{rem}

\begin{cor}
The process of degeneration in $\CP_n(\C)$ defines a partial order on the orbit space
of $n$-dimensional pre-Lie algebra structures, given by $O(\mu)\le O(\la) \iff \mu 
\in \ov{O(\la)}$.
\end{cor}

\begin{proof}
The relation is clearly reflexive. The transitivity follows from the fact that
$O(\la)\subseteq \ov{O(\mu)} \iff \ov{O(\la)}\subseteq \ov{O(\mu)}$.
Finally, antisymmetry follows from the fact, that any orbit is open in its closure.
\end{proof}

The transitivity is very useful. If $\la \ra_{\rm deg} \mu$ and  $\mu \ra_{\rm deg} \nu$,
then $\la \ra_{\rm deg} \nu$.
For $\la\in \CP_n(K)$ we have $\dim O(\la)=\dim \End(V)-\dim \Der(\la)=n^2-\dim \Der(\la)$.
Here $\Der(\la)$ denotes the derivation algebra of the algebra $A$ represented by $\la$.

\begin{cor}
Let $\la\ra_{\rm deg} \mu$ be a proper degeneration in $\CP_n(\C)$. Then $\dim O(\la)>\dim O(\mu)$
and $\dim \Der(\la)<\dim \Der (\mu)$.
\end{cor}

We can represent the degenerations in $\CP_n(K)$ with respect to the above partial order
in a diagram: order the pre-Lie algebras by orbit dimension in $\CP_n(K)$, in each row the
algebras with the same orbit dimension, on top the ones with the biggest orbit dimension. 
Draw a directed arrow between two algebras $A$ and $B$,
if $A$ degenerates to $B$. This diagram is called the {\it Hasse diagram} of degenerations
in $\CP_n(K)$. It shows the classification of orbit closures. \\
A rather trivial example is the case of $1$-dimensional complex pre-Lie algebras.
Let $(e_1)$ be a basis of $\C$. Then there are two pre-Lie algebras. Denote by $P_1$ the
algebra with zero product, and by $P_2$ the algebra with $e_1\cdot e_1=e_1$. Then
$\dim \Der(P_1)=1$ and $\dim \Der(P_2)=0$. 
The Hasse diagram for $\CP_1(\C)$ is given as follows:
$$
\begin{xy}
\xymatrix{
 P_2 \ar[d] \\
 P_1 \\
}
\end{xy}
$$

\section{Criteria for degenerations}

Given two pre-Lie algebras $A$ and $B$ we want to decide whether $A$ degenerates to $B$
or not. Suppose that the answer is yes. Then we would like to find a 
$g_t \in GL_n(\C(t))$ realizing such a degeneration. If the answer is no,
we need an argument to show that such a degeneration is impossible. 
One way is to find an invariant for $A$, i.e., a polynomial in terms of the structure
constants which is zero on the whole orbit of $A$, so that it must be also zero on the
orbit closure of $A$. If $B$ does not satisfy this polynomial equation, then $B$ cannot lie
in the orbit closure of $A$. \\
For example, commutativity of $A$ is such an invariant.
If $L(x)$ resp. $R(x)$ denotes the left resp. right multiplication operators in $\End(A)$, 
then commutativity of $A$ means that the operator $T(x)=L(x)-R(x)$ satisfies
$T(x)=0$ for all $x\in A$. This is clearly such a polynomial invariant on the orbit of $A$. 
Hence a degeneration $A\ra_{\rm deg} B$ is impossible, if $A$ is commutative, but $B$ is not. 
Another operator identity is $T(x,y)=[L(x),R(y)]=L(x)R(y)-R(y)L(x)=0$ for all
$x,y\in A$, which says that the algebra $A$ is associative.

\begin{lem}\label{3.1}
Let $A$ and $B$ be two pre-Lie algebras of dimension $n$, and $T(x_1,\ldots ,x_n)$ be a polynomial
in the operators $L(x_1),\ldots ,L(x_n)$ and  $R(x_1),\ldots ,R(x_n)$.
Suppose that $T(x_1,\ldots ,x_n)=0$ for $A$ but not for $B$. Then $A$ cannot degenerate to $B$.
\end{lem} 

\begin{proof}
Let $\phi\colon A\ra A'$ be an isomorphism of pre-Lie algebras, and denote by $L(x),R(x)$ the
left resp. right multiplications in $A$, and by $\ell(x),r(x)$ the ones in $A'$. Then 
$\phi(x\cdot y)=\phi(x).\phi(y)$ implies
\[
L(x)=\phi^{-1}\kringel \ell (\phi(x))\kringel \phi,\quad 
R(x)=\phi^{-1}\kringel r(\phi(x))\kringel \phi.
\]
If a polynomial $T$ in the left- and right multiplications of $A$ vanishes, then the same
is true for the left-and right multiplications of all pre-Lie algebra structures in the 
$GL_n$-orbit representing $A$, since a base change just induces a conjugation of the 
operator polynomial. It follows that the operator identity holds also for all pre-Lie algebra 
structures in the orbit closure. This implies the claim.
\end{proof}

We can also consider invariants of $\la\in \CP_n(K)$ defining upper (or lower)
semi-continuous functions $f\colon \CL_n(k)\ra \Z_{\ge 0}$. Then  $\la \ra_{\rm deg} \mu$
implies $f(\la)\le f(\mu)$ or $f(\la)\ge f(\mu)$. Recall the following definition:

\begin{defi}
Let $X$ be a topological space. A function $f\colon X\ra \Z_{\ge 0}$
is called {\it upper semi-continuous}, if $f^{-1}(]-\infty,n[)$ is open
in $X$ for all $n \in \Z_{\ge 0}$. It is called {\it lower semi-continuous}, if
$f^{-1}(]n,\infty[)$ is open in $X$ for all $n \in \Z_{\ge 0}$.  
\end{defi}

In the case of Lie algebras, for example, $f(\la)=\dim Z(\la)$ is a upper semi-continuous
function on the variety of $n$-dimensional Lie algebra structures,
and satisfies $f(\la)\le f(\mu)$ for $\la \ra_{\rm deg} \mu$. There are more such
invariants yielding semi-continuous functions, for example the dimensions of cocycle
spaces and Lie algebra cohomology groups. \\
It is very natural to consider similar invariants for pre-Lie algebras $A$.
Define the left and right annihilator, and the center of $A$ by

\begin{align*}
\CL (A) & = \{ x\in A\mid x\cdot A=0 \}, \\
\CR (A) & = \{ x\in A\mid A\cdot x=0 \}, \\
\CZ (A)  & =\{x\in A\mid x\cdot A= A\cdot x=0\}.
\end{align*}

There are also cohomology groups  $H^n_{pre}(A,M)$ for pre-Lie algebras $A$, with
an $A$-bimodule $M$, see \cite{DZU}. The case, where $M=A$ is the regular module goes
already back to Nijenhuis. In this case we have 
$Z^1(A,A)=\Der (A)$, and $Z^2(A,A)$ describes infinitesimal pre-Lie algebra deformations of $A$.
The cohomology of pre-Lie algebras is related to Lie algebra cohomology as follows, see \cite{DZU}:
\[
H^n_{pre}(A,M) \cong H^{n-1}(\Lg_{A},\Hom_K(A,M)).
\]

The various dimensions of these spaces define semi-upper continuous functions:

\begin{lem}
If $A \ra_{\rm deg} B$ then 
\begin{align*}
\dim Z^n_{pre}(A,A) & \le \dim Z^n_{pre}(B,B) \\
\dim H^n_{pre}(A,A) & \le \dim H^n_{pre}(B,B) \\
\dim \CL (A) & \le \dim \CL (B) \\
\dim \CR (A) & \le \dim \CR (B) \\
\dim \CZ(A)   & \le \dim \CZ(B)
\end{align*}
\end{lem}

The proof is similar to the proof in the Lie algebra case. A crucial lemma here is
the following, see \cite{GRH}:

\begin{lem}
Let $G$ be a complex reductive algebraic group with Borel subgroup $B$. If $G$
acts regularly on an affine variety $X$, then for all $x\in X$,
\[
\ov{G\cdot x}=G\cdot (\ov{B\cdot x}).
\]
\end{lem}

We have also the following easy result, which shows that pre-Lie algebra
degenerations in a sense refine the ones for Lie algebras.

\begin{lem}
If $A \ra_{\rm deg} B$ then $\Lg_A \ra_{\rm deg} \Lg_B$ for the associated Lie algebras.
\end{lem}

\begin{proof}
Let $(e_1,\ldots ,e_n)$ be a basis of the underlying vector space. Denote the
product in $A$ by $e_i\cdot e_j$, in $B$ by $e_i.e_j$. The Lie bracktes are given by
$[e_i,e_j]_A=e_i\cdot e_j-e_j\cdot e_i$ and $[e_i,e_j]_B=e_i.e_j-e_j.e_i$. We have
\begin{align*}
\lim_{\ep \to 0} g_\ep ([g_\ep^{-1}(e_i), g_\ep^{-1}(e_j)]_A) & =
\lim_{\ep \to 0} g_\ep \left(g_\ep^{-1}(e_i) \cdot g_\ep^{-1}(e_j) - g_\ep^{-1}(e_j)
 \cdot g_\ep^{-1}( e_i)\right)\\
   & = \lim_{\ep \to 0} g_{\ep} \left(g_\ep^{-1}(e_i) \cdot g_\ep^{-1}(e_j)\right)-
       \lim_{\ep \to 0} g_{\ep} \left(g_\ep^{-1}(e_j) \cdot g_\ep^{-1}(e_i)\right)\\
   & = e_i.e_j-e_j.e_i \\
   & = [e_i,e_j]_B.
\end{align*}
This implies that $\Lg_A \ra_{\rm deg} \Lg_B$, because of remark $\ref{2.8}$.
\end{proof}

We can also generalize the trace invariants of Lie algebras to the case
of pre-Lie algebras. For $x,y\in A$ and $i,j\in \N$ consider the expression
\[
c_{i,j}(A)=\frac{\tr (L(x)^i)\cdot \tr (L(y)^j)}{\tr (L(x)^i L(y)^j)}.
\]
If this is well-defined for all $x,y\neq 0$ and finite, then it is an interesting
invariant of $A$. Just like in the Lie algebra case, we obtain the following result.

\begin{lem}\label{3.6}
Suppose $A$ degenerates to $B$ and both values $c_{i,j}(A)$ and $c_{i,j}(B)$ are
well-defined for all $x,y\neq 0$, then $c_{i,j}(A)=c_{i,j}(B)$.
\end{lem}

We can generalize the definition of pre-Lie algebra derivations as follows. 

\begin{defi}
Let  $\al,\be,\ga \in \C$ and define $\Der_{(\al,\be,\ga)}(A)$ to be the space of all 
$ D\in \End (A)$ satisfying
\[
\al D(x\cdot y)=\be D(x)\cdot y+\ga x \cdot D(y)
\]
for all $x,y \in A$.
We call the elements $D\in\Der_{(\al,\be,\ga)}(A)$ also $(\al,\be,\ga)$-derivations.
\end{defi}

\begin{lem}
If $A\ra_{\rm deg}B$, then
$\dim \Der_{(\al,\be,\ga)}(A) \le \dim \Der_{(\al,\be,\ga)}(B)$ for all $\al,\be,\ga \in \C$.
\end{lem}

\begin{proof}
Let $\la,\mu\in \CP_n(\C)$ represent $A$ and $B$. Fix a basis $(e_1,\ldots ,e_n)$ of the
underlying vector space. Then $\lim_{\ep\to 0}(g_{\ep}\kringel \la)(e_i,e_j)
=\mu(e_i,e_j)$ for operators $g_\ep\in GL_n(\C)$. 
For $D \in \Der_{\alpha , \beta, \gamma} (A)$ we write $D=(d_{ij})_{1 \le i,j \le n}$,
and $D(e_i) = \sum_{l=1}^n d_{li} e_l$.
We have $e_i\cdot e_j = \sum_{k=1}^n c_{ij}^k e_k$ in $A$, with the structure constants
$c_{ij}^k$. Since $D$ is an $(\al,\be,\ga)$-derivation we have
\[ 
\sum_{l=1}^n (\al c_{ij}^l d_{kl} - \be c_{lj}^k d_{li} - \ga c_{il}^k d_{lj}) = 0 
\] 
for all $i,j,k$. We can rewrite these $n^3$ equations as a matrix equation
$Md=0$ where $d$ is the vector formed by the columns of the matrix $D=(d_{ij})$, and $M$
is a $n^3\times n^2$ matrix depending on $c_{ij}^k$ and $\al,\be,\ga$.
Thus we have $\ker(M)=\Der_{(\al,\be,\ga)}(A)$.
If $A$ degenerates to $B$ via $g_{\ep}$ we obtain a sequence of matrices $M_{\ep}$
with $\lim_{\ep\to 0}M_\ep=M_0$ by componentwise convergence of the structure constants,
where $\ker (M_0)=\Der_{(\al,\be,\ga)}(B)$.
Let $m$ be the rank of the matrix $M$. Then every submatrix of size $(m+1)\times (m+1)$
has zero determinant. Because of the convergence the same is true for $M_0$.
It follows that $\rank(M)\ge \rank(M_0)$, or $\dim \ker(M)\le \dim \ker(M_0)$.
\end{proof}

\section{Degenerations in dimension 2}

In this section, we determine the Hasse diagram of degenerations for $2$-dimensional
pre-Lie algebras. This is already quite complicated, and we can apply our results in a 
non-trivial way.
In dimension $n=2$ there are two different complex Lie algebras. Let $(e_1,e_2)$ be a basis.
Then either $\Lg=\C^2$, or $\Lg=\Lr_2(\C)$, where we can choose $[e_1,e_2]=e_1$. 
The classification of $2$-dimensional complex pre-Lie algebras is well known, see
for example \cite{BU1}:
\vspace*{0.5cm}
\begin{center}
\begin{tabular}{c|c|c|c}
$A$ & Products & $\Lg_A$ & $\dim \Der (A)$\\
\hline
$A_1$ & $-$ & $\C^2$  & $4$ \\
\hline
$A_2$ & $e_1\cdot e_1=e_1$ & $\C^2$ & $1$ \\
\hline
$A_3$ & $e_1\cdot e_1=e_1,\; e_2\cdot e_2=e_2$ & $\C^2$ & $0$ \\
\hline
$A_4$ & $e_1\cdot e_2=e_1,\;e_2\cdot e_1=e_1$, & $\C^2$ & $1$ \\
      & $e_2\cdot e_2=e_2$ &  &  \\
\hline
$A_5$ & $e_2\cdot e_2=e_1$ & $\C^2$ & $2$ \\
\hline
$B_1(\al)$ & $e_2\cdot e_1=-e_1,\; e_2\cdot e_2=\al e_2$ & $\Lr_2(\C)$ & $1$ if $\al\neq -1$ \\
           & & & $2$ if $\al=-1$ \\
\hline
$B_2(\be)$ & $e_1\cdot e_2=\be e_1,\; e_2\cdot e_1=(\be-1)e_1$, & $\Lr_2(\C)$ 
& $1$ if $\be\neq 1$ \\
$\be \neq 0$ & $e_2\cdot e_2=\be e_2$ & & $2$ if $\be=1$ \\
\hline
$B_3$  & $e_2\cdot e_1=-e_1,\; e_2\cdot e_2=e_1-e_2$ & $\Lr_2(\C)$ & $1$ \\
\hline
$B_4$  & $e_1\cdot e_1=e_2,\; e_2\cdot e_1=-e_1$ & $\Lr_2(\C)$ & $0$ \\
       & $e_2\cdot e_2=-2e_2$ & & \\
\hline
$B_5$  & $e_1\cdot e_2=e_1,\; e_2\cdot e_2=e_1+e_2$ & $\Lr_2(\C)$ & $1$ \\
\end{tabular}
\end{center}
\vspace*{0.5cm}
Here we have $B_2(0)\simeq B_1(0)$, so that we require $\be \neq 0$.
The pre-Lie algebras $A_1,\ldots ,A_5$ are commutative and associative,
since their Lie algebra  $\Lg=\C^2$ is abelian. The following algebras are Novikov
algebras:
\[
A_1,\, A_2,\, A_3,\, A_4,\, A_5,\, B_2(\be)_{\be \in \C}, \, B_5.
\]
From the list of non-commutative pre-Lie algebras, $B_1(-1)$ and $B_2(1)$ are associative, and 
$B_4$ is simple. \\[0.2cm]
For a pre-Lie algebra $A$ we consider the quantities $c_{i,j}(A)$
for $i,j\in \N$. We have $c_{i,j}(A_2)= 1$, $c_{i,j}(A_4) = 2$, 
$c_{i,j}(B_3) = 2$, $c_{i,j}(B_5) = 1$ for all $i,j\ge 1$, and \\[0.2cm]
\begin{align*}
c_{i,j}(B_1(\al)) & = \frac{(\al^i+(-1)^i)(\al^j+(-1)^j)}{\al^{i+j}+(-1)^{i+j}},\\[0.2cm]
c_{i,j}(B_2(\be)) & = \frac{(\be^i+(\be-1)^i)(\be^j+(\be-1)^j)}{\be^{i+j}+(\be-1)^{i+j}},\\[0.2cm]
c_{i,j}(B_4) & =\frac{(2^i+1)(2^j+1)}{2^{i+j}+1}.\\[0.2cm]
\end{align*}

If $\la \ra_{\rm deg} \mu$ properly, then $\dim O(\la)> \dim O(\mu)$. Therefore, to determine 
the degenerations, we can order the algebras by orbit dimension, i.e., by
$\dim \Der (\la)$, as follows:
\[
A_3,\,B_4,\quad A_2,\,A_4,\,B_1(\al)_{\al\neq -1},\,B_2(\be)_{\neq 0,1},\, B_3,\,B_5,\quad
A_5,\, B_1(-1),\,B_2(1),\quad A_1.
\]

\begin{lem}
The orbit closure of $A_3$ in $\CP_2(\C)$ contains exactly the following algebras:
\[
A_3,\,A_2,\,A_4,\,A_5,\,A_1.
\]
In other words, $A_3$ can only properly degenerate to $A_2,A_4,A_5,A_1$.
\end{lem}

\begin{proof}
First of all, $A_3$ can only degenerate to commutative algebras, see lemma $\ref{3.1}$.
The orbit dimension of $A_3$ is equal to $4$. Hence $A_3$ can only properly degenerate
to commutative algebras of lower orbit dimension, which are exactly the above algebras.
For these we find the following degenerations:
We have $A_3\ra_{\rm deg} A_2$ via 
$g_t^{-1}=\left(\begin{smallmatrix} 1 & 0 \\ t^2 & t \end{smallmatrix}\right)$.
Also, we have a degeneration $A_3\ra_{\rm deg} A_4$ via
$g_t^{-1}=\left(\begin{smallmatrix} 1 & 0 \\ 1 & t \end{smallmatrix}\right)$.
Finally, $A_3\ra_{\rm deg} A_5$ by
$g_t^{-1}=\left(\begin{smallmatrix} 2t^2 & 2t \\ 0 & 1 \end{smallmatrix}\right)$.
\end{proof}

\begin{lem}
The orbit closure of $B_4$ in $\CP_2(\C)$ contains exactly the following algebras:
\[
B_4,\; B_1(-2),\; B_2(-1),\;A_5,\; A_1.
\]
\end{lem}

\begin{proof}
The orbit dimension of $B_4$ equals $4$. Hence $B_4$ can only properly degenerate
to algebras of lower orbit dimension, which are the following ones:
\[
A_2,\;A_4,\;B_1(\al),\;B_2(\be),\;B_3,\;B_5,\;A_5,\; A_1.
\]
Now $B_4$ cannot degenerate to $A_2$ since
\[
c_{1,1}(B_4)=\frac{9}{5}\neq 1=c_{1,1}(A_2).
\]
In the same way, $B_4$ cannot degenerate to $A_4$. \\
Assume that $B_4\ra_{\rm deg}B_1(\al)$. Comparing the invariants $c_{i,j}$ 
for the algebras $B_4$ and $B_1(\al)$ yields that we must have 
$(\al+2)(2\al+1)=0$. For these two values of $\al$ all invariants $c_{i,j}$ coincide, so that 
we cannot exclude that $B_4$ possibly degenerates to $B_1(-2)$, $B_1(-1/2)$.
In fact, there is a degeneration 
$B_4\ra_{\rm deg}  B_1(-2)$ by
$g_t^{-1}=\left(\begin{smallmatrix} t & 0 \\ t & 1 \end{smallmatrix}\right)$.
But there is no degeneration of $B_4$ to any other algebra $B_1(\al)$ for $\al\neq -2$. 
To see this we use lemma $\ref{3.1}$.  
If  $x=x_1e_1+x_2e_2$, $y=y_1e_1+y_2e_2$ then the left and right multiplications of $B_4$
are given by
\begin{align*}
L(x)= & \begin{pmatrix}
-x_2 & 0 \\ x_1 & -2x_1 
\end{pmatrix}, \quad
R(x)=
\begin{pmatrix}
0 & -x_1 \\ x_1 & -2x_1 
\end{pmatrix}, \\
L(y)= & \begin{pmatrix}
-y_2 & 0 \\ y_1 & -2y_1 
\end{pmatrix}, \quad
R(y)=
\begin{pmatrix}
0 & -y_1 \\ y_1 & -2y_1 
\end{pmatrix}. 
\end{align*}
Searching for quadratic operator identities $T(x,y)=0$ for all
$x,y\in B_4$, we find that $T_{r,s}(x,y)=0$ for all $r,s\in \C$, where  \\[0.1cm]
\begin{align*}
T_{r,s}(x,y) & = r (L(x)R(y)-L(y)R(x))+s (R(x)L(y)-R(y)L(x)) \\
       & + (s-3r)[L(x),L(y)] +\frac{1}{2}(r-2s) [R(x),R(y)]. \\
\end{align*}
For $r=s=-2$ we obtain
\[
T(x,y)=[2L(x)-R(x),2L(y)-R(y)]=0. \\
\]
But for $B_1(\al)$ the left and right multiplication operators satisfy this identity if and 
only if $\al=-2$, since in this case
\[
T(x,y)=
\begin{pmatrix}
0 & (\al+2)(x_2y_1-x_1y_2) \\ 0 & 0
\end{pmatrix}.
\]
Hence only $B_4\ra_{\rm deg} B_1(-2)$ is possible.\\
Next assume that $B_4$ degenerates to $B_2(\be)$. A calculation shows that $B_2(\be)$
satisfies the above operator identity if and only if $\be=-1$. We have a degeneration
$B_4 \ra_{\rm deg} B_2(-1)$ however by
$g_t^{-1}=\left(\begin{smallmatrix} \frac{1}{2} & 0 \\ -\frac{1}{2} & t 
\end{smallmatrix}\right)$. \\
The algebra $B_4$ cannot degenerate to $B_3$ since $c_{1,1}(B_4)=9/5\neq 2=c_{1,1}(B_3)$,
and also not to $B_5$ since $c_{1,1}(B_5)=1$. Finally $B_4\ra_{\rm deg} A_5$ by
$g_t^{-1}=\left(\begin{smallmatrix} 2t & 0 \\ t & 3t^2 \end{smallmatrix}\right)$.
\end{proof}

The classification of degenerations among $2$-dimensional pre-Lie algebras is as follows.
We restrict ourselfs to proper degenerations, so that we do not list the algebra itself
in the orbit closure.

\begin{thm}
The orbit closures in $\CP_2(\C)$ are given as follows:
\vspace*{0.5cm}
\begin{center}
\begin{tabular}{c|c}
$A$ & $\ov{O(A)}$ \\
\hline
$A_3$ & $A_2,\,A_4,\,A_5,\,A_1$    \\
$B_4$ & $B_1(-2),\,B_2(-1),\,A_5,\,A_1$   \\
$A_2$ & $A_5,\,A_1$    \\
$A_4$ & $A_5,\,A_1$    \\
$B_1(\al)_{\al\neq -1}$ & $A_5,\,A_1$  \\
$B_2(\be)_{\be \neq 1}$ & $A_5,\,A_1$  \\
$B_3$ & $A_5,\,B_1(-1),\,A_1$ \\
$B_5$ & $A_5,\,B_2(1),\,A_1$ \\
$A_5$ & $A_1$ \\
$B_1(-1)$ & $A_1$ \\
$B_2(1)$ & $A_1$ \\
\end{tabular}
\end{center}
\vspace*{0.5cm}
\end{thm}

\begin{proof}
The classification of the orbit closures for $A_3$ and $B_4$ is given in the two lemmas
above. $A_2$ can only degenerate to commutative algebras of orbit dimension smaller than
$3$, that is to $A_5$ and $A_1$. Both is possible, we have 
$A_2 \ra_{\rm deg} A_5$ by
$g_t^{-1}=\left(\begin{smallmatrix} t & 0 \\ 1 & -t 
\end{smallmatrix}\right)$.
For $A_4$ the same reasoning applies and we have
$A_4 \ra_{\rm deg} A_5$ by
$g_t^{-1}=\left(\begin{smallmatrix} t & 0 \\ 1 & t 
\end{smallmatrix}\right)$. \\
The orbit dimension for $B_1(\al)$, $\al\neq -1$ is $3$, hence possible algebras
in the orbit closure are $A_5$, $B_1(-1)$ and $B_2(1)$. We have a degeneration to $A_5$
by $g_t^{-1}=\left(\begin{smallmatrix} 1 & 0 \\ t & t^2(\al+1) 
\end{smallmatrix}\right)$, for $\al\neq -1$. 
Comparing $c_{1,1}(B_1(\al))=\frac{(\al-1)^2}{\al^2+1}$ and $c_{i,j}(B_1(-1))=2$ 
for $\al^2+1\neq 0$, we see that a degeneration to $B_1(-1)$ is only possible, if
$\al=-1$. But we assumed $\al\neq -1$. Similarly we see that  $B_1(\al)$, $\al\neq -1$
does not degenerate to $B_2(1)$. \\
The only candidates for proper degenerations of the algebras $B_2(\be)$, $\be \neq 1$ 
are again $A_5$, $B_1(-1)$ and $B_2(1)$. There is a degeneration to $A_5$ by
$g_t^{-1}=\left(\begin{smallmatrix} 1 & 0 \\ t & t^2(\be-1) 
\end{smallmatrix}\right)$, for $\be\neq 1$.
For $\be\neq 0,1$ assume that $B_2(\be)\ra_{\rm deg} B_2(1)$. Comparing $c_{1,1}$ we obtain
$\frac{(2\be-1)^2}{\be^2+(\be-1)^2}=1$, or equivalently $\be(\be-1)=0$ which was excluded.
Similarly $B_2(\be)$, ${\be \neq 1}$ cannot degenerate to $B_1(-1)$. Another possibility to
show this is to use lemma $\ref{3.6}$. \\
$B_3$ degenerates to $A_5$ by 
$g_t^{-1}=\left(\begin{smallmatrix} t^{-2} & 0 \\ 0 & t^{-1} 
\end{smallmatrix}\right)$, and to $B_1(-1)$ by
$g_t^{-1}=\left(\begin{smallmatrix} -t & 0 \\ 0 & 1 
\end{smallmatrix}\right)$. Because $c_{1,1}(B_3)=2$ and  $c_{1,1}(B_2(1))=1$, there is
no degeneration from $B_3$ to $B_2(1)$. \\
$B_5$ degenerates to $A_5$ by 
$g_t^{-1}=\left(\begin{smallmatrix} t^{-2} & 0 \\ 0 & t^{-1} 
\end{smallmatrix}\right)$, and to $B_2(1)$ by 
$g_t^{-1}=\left(\begin{smallmatrix} t^{-1} & 0 \\ 0 & 1 
\end{smallmatrix}\right)$. Because $c_{1,1}(B_5)=1$ and  $c_{1,1}(B_1(-1))=2$, there is
no degeneration from $B_5$ to $B_1(-1)$.
\end{proof}

\begin{cor}
The Hasse diagram of degenerations in $\CP_2(\C)$ is given as follows:\\[0.2cm]
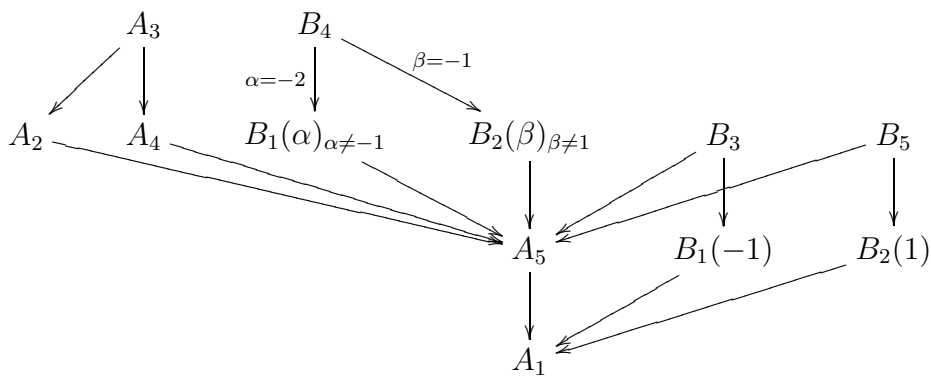
$$
\begin{xy}
\xymatrix{
 &  A_3 \ar[ld] \ar[d] & B_4 \ar[d]_{\al=-2} \ar[rd]^{\be=-1} & & & \\
A_2\ar[rrrd]  &  A_4\ar[rrd]  & B_1(\al)_{\al\neq -1}\ar[rd] & B_2(\be)_{\be\neq 1}\ar[d]
 & B_3 \ar[d] \ar[ld] & B_5\ar[d] \ar[lld]\\
 & & & A_5 \ar[d] & B_1(-1) \ar[ld] & B_2(1) \ar[lld] \\
 & & & A_1 & & \\
}
\end{xy}
$$

\end{cor}

\begin{cor}
The Hasse diagram for degenerations of Novikov algebra structures in $\CP_2(\C)$ is given as 
follows:\\[0.2cm]
$$
\begin{xy}
\xymatrix{
 &  A_3 \ar[ld] \ar[d] & & \\
A_2\ar[rd]  &  A_4\ar[d]  & B_2(\be)_{\be\neq 1}\ar[ld] & B_5 \ar[ld] \ar[lld]\\
  & A_5\ar[d] &  B_2(1)\ar[ld] & \\
 & A_1 & &
}
\end{xy}
$$
\end{cor}

\end{document}